\documentclass[11pt,a4paper]{article}
\usepackage{mathrsfs}
\usepackage{enumerate}
\usepackage{amsfonts}
\usepackage[active]{srcltx}
\usepackage{indentfirst,latexsym,bm,amsmath,pstricks,amssymb,amsthm,amscd}
\usepackage[all]{xy} %%lu
\usepackage[all]{xy}
\usepackage{tikz,diagbox}
\usepackage{makecell}
\begin{document}
%\addtolength{\parinddent}{2em plus 1em minus 1em]
%======================= ±êÌâÃû³ÆÖÐÎÄ»¯ ============================%
\theoremstyle{plain}
\newtheorem{thm}{Theorem}[section]
\newtheorem{theorem}[thm]{Theorem}
\newtheorem{addendum}[thm]{Addendum}
\newtheorem{lemma}[thm]{Lemma}
\newtheorem{corollary}[thm]{Corollary}
\newtheorem{proposition}[thm]{Proposition}
%%%%%%%%%%%%%%%%%%%% Text roman %%%%%%%%%%%%%%%%%%%%%%%%%%%%%

\newcommand{\mR}{\mathbb{R}}
\newcommand{\mZ}{\mathbb{Z}}
\newcommand{\mN}{\mathbb{N}}
\newcommand{\bC}{\mathbb{C}}
\newcommand{\mP}{\mathbb{P}}
\newcommand{\CO}{\mathcal{O}}
\newcommand{\CE}{\mathcal{E}}
\newcommand{\CF}{\mathcal{F}}
\newcommand{\CG}{\mathcal{G}}
\newcommand{\CL}{\mathcal{L}}
\newcommand{\CM}{\mathcal{M}}
\newcommand{\CP}{\mathcal{P}}
\newcommand{\CS}{\mathcal{S}}
\newcommand{\CA}{\mathcal{A}}
\newcommand{\CB}{\mathcal{B}}
\newcommand{\CC}{\mathcal{C}}
\newcommand{\CH}{\mathcal{H}}
\newcommand{\CI}{\mathcal{I}}
\newcommand{\CJ}{\mathcal{J}}
\newcommand{\CZ}{\mathcal{Z}}
\newcommand{\CN}{\mathcal{N}}
\newcommand{\CR}{\mathcal{R}}
\newcommand{\CT}{\mathcal{T}}

\newcommand{\FA}{\mathfrak{A}}
\newcommand{\FB}{\mathfrak{B}}
\newcommand{\FC}{\mathfrak{C}}
\newcommand{\FD}{\mathfrak{D}}
\newcommand{\FM}{\mathfrak{M}}
\newcommand{\FT}{\mathfrak{T}}

\newcommand{\SA}{\mathscr{A}}
\newcommand{\SB}{\mathscr{B}}
\newcommand{\SC}{\mathscr{C}}
\newcommand{\SD}{\mathscr{D}}

\newcommand{\mr}{\mbox}
\newcommand{\I}{{\bf I}}
\newcommand{\J}{{\bf J}}
\newcommand{\bs}{{\bf s}}
\newcommand{\ts}{{\tilde s}}
\newcommand{\e}{{{\bf e}}}
\newcommand{\m}{{\bf m}}
\newcommand{\ord}{\mathrm{ord}}
\newcommand{\fb}{\mathfrak{b}}

\def\calm{\mathcal M}
\def\cals{\mathcal S}
\def\calh{\mathcal H}
\def\cala{\mathcal A}
\def\calt{\mathcal T}
\def\lra{\longrightarrow}
\def\calth{\mathcal{TH}}

\theoremstyle{definition}
\newtheorem{notations}[thm]{Notations}
\newtheorem{definition}[thm]{Definition}
\newtheorem{claim}[thm]{Claim}
\newtheorem{assumption}[thm]{Assumption}
\newtheorem{assumptions}[thm]{Assumptions}
\newtheorem{property}[thm]{Property}
\newtheorem{properties}[thm]{Properties}
\newtheorem{example}[thm]{Example}
\newtheorem{examples}{Examples}
\newtheorem{conjecture}[thm]{Conjecture}
\newtheorem{questions}[thm]{Questions}
\newtheorem{question}[thm]{Question}
\newtheorem{problem}[thm]{Problem}

\theoremstyle{remark}
\newtheorem{remark}[thm]{Remark}
\newtheorem{remarks}[thm]{Remarks}

\numberwithin{equation}{section}
 \newcommand{\Rnm}[1]{\uppercase\expandafter{\romannumeral #1}}

\newcommand{\xyp}[1]{\begin{eqnarray*}
\xymatrix{#1}
\end{eqnarray*}} %%lu

 \title{Slopes of fibrations with trivial vertical fundamental groups}
\author{Xiao-Lei Liu~~~~Xin Lu}

%\footnotetext[1]{}
\footnotetext[1]{ \ The first author is supported by NSFC (No. 12271073); the second author is supported by National Natural Science Foundation of China,  Fundamental Research Funds for central Universities, and Science and Technology Commission of Shanghai Municipality (No. 22DZ2229014).}

\footnotetext[2]{ \ {\itshape 2010 Mathematics Subject
Classification.}  14D06, 14H10, 14H30 }

\footnotetext[3]{\ {\itshape Key words and phrases.}
Fundamental group, Kodaira fibration,  slope,   modular invariant.}

\date{}

 \maketitle

\begin{abstract}
	Kodaira fibrations  have  non-trivial vertical fundamental groups and their slopes are all 12.
	In this paper,  we  show that 12 is indeed the sharp upper bound for the slopes of fibrations with trivial vertical fundamental groups.
	Precisely, for each  $g\geq3$ we prove the existence of fibrations of genus $g$ with trivial vertical fundamental groups whose slopes can be arbitrarily close to 12.
	This gives a relative analogy of Roulleau-Urz\'ua's work \cite{RU15} on the slopes of surfaces of general type with trivial fundamental groups.
		%As an application of our method, we answer the geography  problem of slopes of fibrations completely, where the problem has been solved in \cite{LL} except for the case $g=3$.
\end{abstract}

%\linenumbers

\section{Introduction}
We work over the complex number $\mathbb{C}$.
Let $X$ be a smooth minimal projective surface of general type.
Denote by $c_1^2(X),c_2(X)$ the two Chern numbers of $X$.
It is well-known that $c_1^2(X)=K_X^2$ is the self-intersection number of the canonical divisor $K_X$,
and that $c_2(X)=\chi_{\mathrm{top}}(X)$ is the Euler number of $X$.
Moreover, they satisfy the following Noether equality:
\begin{align}\label{eqnNoet}
  c_1^2(X)+c_2(X)=K_X^2+\chi_{\mathrm{top}}(X)=12\chi(\CO_X),
\end{align}
where $\chi(\CO_X)$ is the Euler characteristic of the structure sheaf $\CO_X$.
A fundamental problem in the study of the geography of surfaces of general type is:
for which pair of integers $(x,y)$,
there exists a smooth minimal projective surface $X$ of general type such that $\big(c_1^2(X),c_2(X)\big)=(x,y)$?
By \eqref{eqnNoet}, this is equivalent to describe all the possible values
$\big(K_X^2,\chi(\CO_X)\big)$.
There is a long history  to this problem and we refer to  \cite{Pe81,Pe87} and \cite[VII \S 8]{BHPV04}  for an introduction.
First, both $K_X^2$ and $\chi(\CO_X)$ are positive integers and satisfy the following restrictions (Noether's inequality and Miyaoka-Yau's inequality):
$$2\chi(\CO_X)-6\leq K_X^2 \leq 9\chi(\CO_X).$$
Sommese \cite{So84} proved that every rational number between $2$ and $9$ can be
realized as the quotient $K_X^2 /\chi(\CO_X)$
(will be called the slope of $X$ for convenience) of some surface $X$.

The situation becomes subtle when the simple connectedness condition is imposed on the surface $X$.
Persson \cite{Pe81} constructed a series of examples showing that all the rational numbers between $2$ and $8$ occur as the slopes of simply connected surfaces. According to the characterization of the Miyaoka-Yau equality,
any surface of general type with $K_X^2=9\chi(\CO_X)$ is a ball quotient and hence can not be simply connected.
Simply connected surfaces of general type with nonnegative index (i.e., its slope $\geq 8$)
seemed difficult to construct.
In fact, the Watershed conjecture predicted that any simply connected surface of general type has negative index (i.e., its slope $<8$).
Nevertheless, Moishezon-Teicher \cite{mt87} constructed the first example
of simply connected surfaces with positive index,
and Chen \cite{ch87} found simply connected surfaces with slopes up to $8.75$.
More recently in the beautiful work \cite{RU15}, Roulleau-Urz\'ua proved that the slopes of simply connected surfaces are dense in $[8,9)$.
In particular, the slopes of
simply connected surfaces of general type can be arbitrarily close to $9$.

\vspace{1mm}
We are interested in the analogy of the geography problem for surface fibrations.
Let $f\colon  X \to C$ be a relatively minimal surface fibration whose general fiber is of genus $g\geq 2$.
We consider the following relative invariants:
\begin{align*}
	\begin{cases}
		K_f^2=K_X^2-8(g-1)(g(C)-1),\\
		\chi_f=\chi(\mathcal O_X)-(g-1)(g(C)-1),\\
		e_f=\chi_{\mathrm{top}}(X)-4(g-1)(g(C)-1).
	\end{cases}
\end{align*} %%lu
These invariants are non-negative integers satisfying the following properties:\vspace{-1mm}
	\begin{enumerate}
	\item[(1).]  $K_f^2=0 ~\Longleftrightarrow ~\chi_f=0$, if and only if $f$ is locally trivial.
	\item[(2).] $e_f=0$ if and only if $f$ is smooth.
	\item[(3).] The Noether equality holds: $12\chi_f=K_f^2+e_f$.
\end{enumerate}

\vspace{-1mm}
Analogously, one may ask a similar geography problem for surface fibrations (cf. \cite[\S\,1.1]{AK02}):
{\it for which pair of integers $(x,y)$,
there exists a relatively minimal surface fibration $f\colon X \to C$
such that $\big(K_f^2,\chi_f\big)=(x,y)$}?
For a locally non-trivial fibration $f$, the
{\it slope} of $f$ is defined as $\lambda_f=K_f^2/\chi_f$,
and it satisfies:
$$\frac{4(g-1)}{g} \leq \lambda_f \leq 12 .$$
The first inequality is the slope inequality \cite{Xi87,CH88},
and the equality can be reached by hyperelliptic fibrations;
the second inequality follows from the non-negativity of $e_f$,
and the equality can be reached by Kodaira fibrations whenever $g\geq 3$.

Xiao started to consider the above geography problem for surface fibrations around 1980s,
where he got partial answers for fibrations of genus $g=2$, cf. \cite[Theorem 2.9]{Xi85} and \cite[Theorem 4.3.5]{Xi92}.
Motivated by Xiao's work, Chen \cite{ch87} generalized it to hyperelliptic fibrations,
based on which Chen constructed many simply connected surfaces with positive index as mentioned before.
Recently, Liu-Lu \cite{LL} proved that any rational number $r\in \left[4(g-1)/g,\,12\right]$ can be realized as the slope of a  fibration of genus $g$ whenever $g>3$,
which gives an analogy of Sommese's result \cite{So84} in the relative version.

We would like to take the fundamental group into consideration.
Recall that the fundamental group of the fibered surface $X$ can be divided into two parts with the following exact sequence (cf. \cite{Xi91}):
$$1 \lra \mathcal{V}_f \lra \pi_1(X) \longrightarrow \mathcal{H}_f \longrightarrow 1,$$
where $\mathcal{V}_f$ is called the {\it vertical fundamental group} of the fibration $f$, see Section \ref{fundgp} for a more precise definition.
Xiao \cite{Xi87} proved that the vertical fundamental group $\mathcal{V}_f$ is trivial if $f$ is non-hyperelliptic with $\lambda_f<4$.
It is not difficult \cite{Xi91} to construct surface fibrations with lower slope admitting a trivial vertical fundamental group.
The situation is different if the slope $\lambda_f$ is large.
In the extreme case when $\lambda_f=12$, the vertical fundamental group $\mathcal{V}_f$ can never be trivial, cf. Lemma \ref{lem-2-1}.
Hence it is natural to wonder: is $12$  the sharp upper bound for
the slopes of fibrations with trivial vertical fundamental groups?
Our main purpose is to answer  this question  affirmatively,
which gives a relative analogy of Roulleau-Urz\'ua's work \cite{RU15} on the slopes of surfaces of general type with trivial fundamental groups.

 \begin{theorem}\label{thm1}
	For each $g\geq3$, there exists a sequence of fibrations $f_n\colon  X_n \to C_n$ of genus $g$
	such that $\mathcal{V}_{f_n}=\{1\}$ and $\lim\limits_{n\to\infty}\lambda_{f_n}=12.$
\end{theorem}

Different from the constructions (using cyclic coverings) in \cite{ch87,LL},
our construction is based on the moduli spaces and the Torelli map.
As a byproduct of our method, we can construct more examples with given slopes in the case of $g=3$.
 \begin{theorem}\label{thm2}
	For each rational number $r\in [\frac83,12)$, there exists a fibration $f$ of genus $3$ such that $\mathcal{V}_f=\{1\}$ and $\lambda_f=r$.
\end{theorem}
We tend to believe this should be true for each $g>3$; see Remark \ref{rmk-gen}.
Combining with the results in \cite{LL}, we get the following.
\begin{corollary}
	For each $g\geq2$ and each rational number $r\in [4(g-1)/g, \lambda_M(g)]$, there exists a fibration of genus $g$ with slope $r$, where
	$$
	\lambda_M(g):=
	\begin{cases}
		7, & \mbox{if~} g=2,\\
		12, & \mbox{if~} g\geq 3.
	\end{cases}
	$$
\end{corollary}

\begin{remark}
(1). If $f$ is hyperelliptic, it is known \cite{Xi92} that $\lambda_f\leq\lambda_M^h(g)$
and the upper bounds can be reached (see \cite{Mo98,LT13}),
where
$$
\lambda_M^h(g):=
\begin{cases}
	12-\frac{8g+4}{g^2}, & \mbox{if~} g \mbox{~is~even,}\\
	12-\frac{8g+4}{g^2-1}, & \mbox{if~} g \mbox{~is~odd.}
\end{cases}
$$
When $g=2$, the fibration $f$ is always hyperelliptic, and hence $\lambda_f\leq 7$ in this case.
It is well-known that there exist Kodaira fibrations of genus $g$ for every $g\geq 3$.
Hence $\lambda_M(g)$ is indeed the sharp upper bound of the slopes of fibrations of genus $g$.

(2). The hyperelliptic fibrations constructed in \cite{LL} are in fact with trivial vertical fundamental group.
In other words, these examples show that for each genus $g\geq2$ and each rational number
$r\in \big(4(g-1)/g, \lambda_M^h(g)\big)$, there exists a fibration $f$ of genus $g$ satisfying that $\mathcal{V}_f=\{1\}$ and $\lambda_f=r$, see Remark \ref{remark}.
However, the vertical fundamental groups of non-hyperelliptic fibrations
constructed in \cite{LL} seem far from being trivial.
\end{remark}

\section{Preliminaries}\label{sec-pre}
\subsection{Moduli spaces and the Torelli map}
Let $\CM_g$ (resp. $\CH_g$) be the moduli space of smooth projective curves (resp. hyperelliptic curves) of genus $g\geq 2$, and $\CA_g$ be the moduli space of $g$-dimensional principally polarized abelian varieties. Recall that the Torelli morphism
$$j\colon \CM_g\lra\CA_g,$$
which associates to a curve its Jacobian with the canonical principal polarization.
The Torelli morphism $j$ is injective; in fact it is an immersion \cite{OS79}.
Let $\overline{\CM_g}=\overline{\CM_g}^{DM}$ be the Deligne-Mumford compactification of $\calm_g$,
and $\overline{\cala_g}=\overline{\cala_g}^{S}$ be the Satake compactification of $\cala_g$.
The boundary $\overline{\CM_g}\setminus \calm_g$ consists of several divisors $\Delta_i$ ($0\leq i\leq [g/2]$);
while the boundary $\overline{\cala_g}\setminus \cala_g$ is of codimension at least two.
In fact,
$$\overline{\cala_g}\setminus \cala_g=\bigcup\limits_{0\leq i \leq g-1} \cala_i.$$
It turns out that the Torelli map $j$ extends to a map from $\overline{\calm_g}$ to $\overline{\cala_g}$,
which is still denoted by $j$.
The extended Torelli map is no longer injective:
it collapses the boundary divisors.
When $g\geq 3$, the boundary $j(\overline{\calm_g}) \setminus j(\calm_g)$ is of codimension at least two,
cf. \cite{dm18}.

In order to assure the representability, it is necessary to take certain level structure into consideration.
Fixing $l\geq 3$ an integer, let  $\calm_{g,[l]}$ (resp. $\cala_{g,[l]}$ and so on) be the corresponding moduli space with full level-$l$ structure.
No specific choice of the level $l(\geq3)$ is made because it is only imposed to assure the representability,
which plays no essential role in our study.
The Torelli map $j$ can be similarly defined over the moduli spaces with level structure.
However, the Torelli map
$$j\colon  \calm_{g,[l]} \lra \cala_{g,[l]}$$
is no longer an immersion, but a two-to-one map ramified exactly over the locus of hyperelliptic curves $\calh_{g,[l]} \subseteq \calm_{g,[l]}$, cf. \cite{OS79}.
As the level $l\geq 3$, there is a universal family of smooth curves (resp. stable curves)
over $\calm_g$ (resp. $\overline{\calm_g}$) with the following commutative diagram, cf. \cite[Theorem\,10.9]{po77}.
$$\xymatrix{\mathcal{S}_{g,[l]} \ar@{^(->}[rr] \ar[d]^-{\mathfrak{f}}
	&&  \overline{\mathcal{S}_{g,[l]}}\ar[d]^-{\bar{\mathfrak{f}}}\\
	\calm_{g,[l]} \ar@{^(->}[rr] && \overline{\calm_{g,[l]}}
}$$

The following lemma can be found in \cite[Lemma\,A.1]{LZ19}.

\begin{lemma}\label{existprop}
	There exists an involution $\sigma_g$ {\rm(}resp. $\tau_g${\rm)}
	on $\cals_{g,[l]}$ {\rm(}resp. $\calm_{g,[l]})$
	such that the following diagram commutes.
	$$\xymatrix{
		\cals_{g,[l]} \ar[d]_{\mathfrak f}\ar[rr]^{\sigma_g} && \cals_{g,[l]} \ar[d]^{\mathfrak f}\\
		\calm_{g,[l]} \ar[rr]^{\tau_g}  && \calm_{g,[l]}
	}
	$$

\noindent	Moreover, $j\circ \tau_g(x)=j(x)$ for any $x\in \calm_{g,[l]}$,
the fixed locus of $\tau_g$ is exactly the hyperelliptic locus $\calh_{g,[l]}\subseteq \calm_{g,[l]}$,
	and for $p\in \calh_{g,[l]}$, $\sigma_g|_{F_p}\colon F_p \to F_p$ is the hyperelliptic involution of $F_p$,
	where $F_p \subseteq \cals_{g,[l]}$ is the hyperelliptic curve over $p$.
\end{lemma}

\subsection{Modular invariants}\label{sec-modular}

Let $f\colon X\to C$ be a surface fibration (or simply fibraion),
i.e., $f$ is a proper surjective morphism from a smooth projective surface onto a smooth projective curve with connected fibers.
Denote by $g$ the genus of a general fiber of $f$.
We will always assume that $f$ is relatively minimal,
i.e., there is no $(-1)$-curve contained in fibers of $f$.
If $\pi\colon \tilde C \to C$ is a base change of degree $d$, then the {\it pullback fibration} $\tilde f\colon \tilde X\to \tilde C$ of $f$ with respect to $\pi$ is defined as the relatively minimal  model  of the desingularization of $X\times_C\tilde C\to \tilde C$.
Denote by $F$ the singular fiber of $f$ over
$p=f(F)\in C$. If $\pi$ is totally ramified over $p$, then the fiber $\tilde F$ of the pullback fibration $\tilde{f}$ over $\pi^{-1}(p)$  is called {\it $d$-th root model} of $F$. If $\pi$ is ramified over $p$ and some non-critical points of $f$ such that the fibers of $\tilde f$ over $\pi^{-1}(p)$ are all semistable, then the {\it Chern numbers} of $F$ are defined as follows (see \cite{Ta94,Ta96}),
\begin{equation*}
c_1^2(F)=K_f^2-\frac1dK_{\tilde{f}}^2,~~c_2(F)=e_f-\frac1de_{\tilde{f}},~~\chi_F=\chi_f-\frac1d\chi_{\tilde{f}}.
\end{equation*}

The fibration $f$ induces a moduli map $J\colon C\to \overline{\mathcal
M_g}$ from $C$ to the moduli space $\overline{\mathcal
M_g}$. Let $\lambda$ be the Hodge divisor class of $\overline{\mathcal M_g}$, $\delta$ be the boundary divisor class, and  $\kappa
=12\lambda-\delta$ be the first Morita-Mumford class. Then there are three fundamental {\it modular invariants} of $f$
 defined as follows (see \cite{Ta10}),
$$
\kappa(f)=\deg J^*\kappa , \hskip0.3cm \lambda(f)=\deg J^*\lambda ,
\hskip0.3cm \delta(f)=\deg J^*\delta.
$$
If $f$ is semistable, then
\begin{equation}\label{eqss}
\kappa(f)=K_{{f}}^2,~~\delta(f)=e_{{f}}, ~~ \lambda(f)=\chi_{{f}}.
\end{equation}
These modular invariants satisfy the {\it base change property}, i.e., if $\tilde f$ is the pullback
fibration  of $f$ with respect to a base change of degree $d$,  then
\begin{equation}\label{timesd}
\kappa(\tilde f)=d\cdot \kappa(f),~~\delta(\tilde f)=d\cdot
\delta(f), ~~ \lambda(\tilde f)=d\cdot \lambda(f).
\end{equation}
 It is proved in \cite{Ta94, Ta96} that
\begin{equation}\label{modinv}
\begin{cases}
K_f^2=\kappa(f)+\sum_{i=1}^sc_1^2(F_i),&\\
e_f=\delta(f)+\sum_{i=1}^sc_2(F_i),&\\
\chi_f=\lambda(f)+\sum_{i=1}^s\chi_{F_i},
\end{cases}
\end{equation}
where  $F_1,\cdots,F_s$ are all the singular fibers of $f$.

\begin{example}\label{exa:FhFg}
	Let $F^h_g$  be the singular hyperelliptic fiber of genus $g$ with the following dual graph, where  $\overset{n}{\underset{-e}{\circ}}$
denotes a smooth rational curve with self-intersection number $(-e)$
and  multiplicity $n$ in $F^h_g$. For convenience, we omit the subscript $(-e)$  whenever  $e=2$.
		{\upshape
		\begin{center}
			\begin{tikzpicture}
				[  place/.style={circle,draw,inner sep=0.5mm},
				place2/.style={circle,draw,fill,inner sep=0.2mm},]
				
				\node[place] (v1) at (0,1.5) [label=above:2] [label=left:{$C_0$}]  {};
				\node[place] (v2)  at (2,0) [label=above:1]  {};
				\node[place] (v3) at (2,2)  [label=above:1]  {};
				\node[place] (v4) at (2,3)  [label=above:1]  {};
				
				\node[place2]   at (2,1) {};
				\node[place2]   at (2,0.7) {};
				\node[place2] at (2,1.4)    {};
				\node[place2]  at (2,1.7)    {};
				
				\node[right=2pt] at (v2)  { ${C_{2g+2}}$};
				\node[right=2pt] at (v3)   { ${C_2}$};
				\node[right=2pt] at (v4)   {${C_1}$};
				\node[below=7pt] at (v1)   {\scriptsize $-(g+1)$};
				
				\draw (v1)--(v2); \draw (v1)--(v3); \draw (v1)--(v4);
				\node at (0.5,0) [below=8pt] {Figure 1: Hyperelliptic fiber $F^h_g$.};
			\end{tikzpicture}
	\end{center}}
	\vspace{-4mm}
\noindent
Then $F^h_g$ is simply connected \cite[Lemma A]{Pe81},
and the Chern numbers of $F^h_g$ are as follows (see also \cite[Example 2.4]{LL})
	\begin{align}\label{eqnchern}
		c_1^2(F^h_g)=2g-2, ~~\chi_{F^h_g}=\frac{g}2.
	\end{align}

 Let $\tilde F^h_g$ be the $d$-th root model of $F^h_g$. Since
 $$F_g^h=\sum_{i=0}^{2g+2}n_iC_i=2C_0+C_1+\cdots+C_{2g+2}$$
  is normal crossing, by \cite[Section 2.2]{LuT13}, the multiplicity of the strict transform of $C_i$ in $\tilde F^h_g$ is $n_i/\gcd(n_i,d)$ and $\tilde F^h_g=F^h_g$ when $d$ is odd.
\end{example}

\subsection{The fundamental group}\label{fundgp}
In this subsection,
we recall some general facts about the fundamental group of a surface fibration.
Let $f\colon  X \to C$ be a fibration of genus $g$, and $F$ be any general fiber.
According to \cite[Lemma 1]{Xi91}, the embedding of $F$ in $X$ induces a homomorphism
$$\alpha\colon  \pi_1(F) \lra \pi_1(X),$$
whose image is a normal subgroup of $\pi_1(X)$ and independent of the choice of the general fiber $F$.
The image of $\alpha$, denoted by $\mathcal{V}_f$, is called {\it the vertical fundamental group} of $X$.
Let $\mathcal{H}_f:=\pi_1(X)/\mathcal{V}_f$ be the quotient group,
then one obtains the following exact sequence:
$$1 \lra \mathcal{V}_f \lra \pi_1(X) \longrightarrow \mathcal{H}_f \longrightarrow 1.$$
The next property is more or less well-known.
\begin{lemma}\label{lem-2-1}
	Suppose that $f$ is a Kodaira fibration of genus $g\geq 3$,
	then the vertical fundamental group $\mathcal{V}_f$ is non-trivial.
\end{lemma}
\begin{proof}
	In this case, $f$ admits no singular fiber, and hence it is a topological fiber bundle.
	It follows that the induced homomorphism from $\pi_1(F)$ to $\pi_1(X)$ is injective with the following exact sequence:
	$$1 \lra \pi_1(F) \overset{\alpha}{\lra} \pi_1(X) \lra  \pi_1(C) \to 1.$$
	In particular, the vertical fundamental group $\mathcal{V}_f$ is non-trivial.
\end{proof}
The next lemma, due to Xiao \cite[Lemma\,3]{Xi91}, provides a useful way to determine
whether $\mathcal{V}_f=\{1\}$.
\begin{lemma}\label{lem-2-2}
	Let $F_0$ be any fiber of $f\colon X \to C$ and $\alpha_0\colon \pi_1(F_0) \to \pi_1(X)$ the induced homomorphism, 	then $\mathcal{V}_f \subseteq \mathrm{Im}(\alpha_0)$. In particular, $\mathcal{V}_f$ is trivial if $f$ admits a simply connected fiber.
\end{lemma}

\begin{remark}\label{remark}
	It is shown in \cite{LL} that for each $g\geq2$ and each rational number $r\in \big(4(g-1)/g,\lambda_M^h(g)\big)$,
	there exists a hyperelliptic fibration $f_{g,r}$ of genus $g$ with $\lambda_{f_{g,r}}=r$. By the construction of $f_{g,r}$ in the proof of \cite[Theorem 1.3]{LL}, we know that $f_{g,r}$ has at least one singular fiber $F^h_g$, which is simply connected by Example \ref{exa:FhFg}.
	Hence the vertical fundamental group of $f_{g,r}$ is trivial by Lemma \ref{lem-2-2}.
\end{remark}

\section{The construction}
In this section, we mainly aim to construct surface fibrations with trivial vertical fundamental groups whose slopes tend to $12$.
 \begin{proof}[Proof of Theorem \ref{thm1}]
 	We use the notations and terminology introduced in Section \ref{sec-pre}.
 	Let $L_1,L_2,\ldots,L_{3g-4}$ be  very ample divisors on the Satake compactification
 	$\overline{\cala_{g,[l]}}=\overline{\cala_{g,[l]}}^{S}$,
 	where the level $l\geq 3$.
 	Since $\dim \calm_g=3g-3$ and the boundary $j(\overline{\calm_g}) \setminus j(\calm_g)$ is of codimension at least two,
 	the intersection $L_1L_2\cdots L_{3g-4}$ is a complete smooth curve  $C \subseteq j(\CM_g)$.
 	Moreover, by the very ampleness of these $L_i$'s, we may assume that the curve $C$
 	intersects $j(\calh_{g,[l]})$ transversely with $C \cap j(\calh_{g,[l]}) \not=\emptyset$.
 	 	
 	Let $B=j^{-1}(C)\subseteq \calm_{g,[l]}$ be the inverse image of $C$.
 	Note that the Torelli morphism
 	$$j\colon \CM_{g,[l]}\lra\CA_{g,[l]}$$
 	is two-to-one and ramified exactly over the hyperelliptic locus $\calh_{g,[l]}$.
%   As $C$
% 	intersects $j(\calh_{g,[l]})$ transversely with $C \cap j(\calh_{g,[l]}) \not=\emptyset$, we know that
 	As $C$
 	intersects $j(\calh_{g,[l]})$ transversely with $C \cap j(\calh_{g,[l]}) \not=\emptyset$,
 	it follows that
 	$B$ is smooth and irreducible with $B \cap \calh_{g,[l]} \not=\emptyset$.
 	Let $h\colon X \to B$ be the  family of smooth curves induced from the universal family $\mathfrak{f}\colon  \mathcal{S}_{g,[l]} \to \calm_{g,[l]}$.
 	According to Lemma \ref{existprop}, there is an involution $\sigma$ (resp. $\tau$)
 	on $X$ (resp. $B$)
 	such that the following diagram commutes.
 	$$\xymatrix{
 		X \ar[d]_{h}\ar[rr]^{\sigma} && X \ar[d]^{h}\\
 		B \ar[rr]^{\tau}  && B
 	}
 	$$
 	Moreover, the quotient $B/\langle\tau\rangle \cong C$,
 	and $\sigma|_{\Gamma_p}\colon \Gamma_p \to \Gamma_p$ is the hyperelliptic involution of $\Gamma_p$ for any fixed point $p\in \mathrm{Fix}(\tau)\subseteq B$,
 	where $\Gamma_p=h^{-1}(p)$ is the fiber over $p$.
 	
 	Let $f_0\colon X_0 \to C\cong B/\langle\tau\rangle$ be the relatively minimal model of
 	the quotient fibred surface $X/\langle\sigma\rangle \to C \cong B/\langle\tau\rangle$.
 	By our construction, the possible singular fibers of $f_0$
 	are those corresponding to $\Gamma_p$ with $p\in \mathrm{Fix}(\tau)\subseteq B$.
 	As the restricted map
 	$$\sigma|_{\Gamma_p}\colon \Gamma_p \to \Gamma_p$$
 	is the hyperelliptic involution of $\Gamma_p$,
 	its image $\Pi(\Gamma_p)$ is a rational curve with $2g+2$ singularities of $X/\langle\sigma\rangle$
 	on $\Pi(\Gamma_p)$, where $\Pi\colon X \to X/\langle\sigma\rangle$ is the quotient map.
 	By resolving the singularities, it follows that the corresponding fiber $F_{j(p)}=f_0^{-1}\big(j(p)\big)$ in $X_0$
 	is isomorphic to $F_g^h$, which is described in Example \ref{exa:FhFg} for any $p\in \mathrm{Fix}(\tau)$.
 	Let $F_1,\cdots,F_s$ be all the singular fibers of $f_0$,
 	then by \eqref{modinv} and \eqref{eqnchern} we have that
 	$$\begin{aligned}
 	K_{f_0}^2&\,=\kappa(f_0)+\sum_{i=1}^{s} c_1^2(F_i)=\kappa(f_0)+(2g-2)s,\\
 	\chi_{f_0}&\,=\lambda(f_0)+\sum_{i=1}^{s}\chi_{F_i}=\lambda(f_0)+\frac12gs.
 	\end{aligned}$$
 	Note that $h$ is a Kodaira fibration.
 	By \eqref{eqss} and \eqref{timesd},
 	$$\kappa(f_0)=\frac12\kappa(h)=\frac12K_h^2,
 	\quad\text{and}\quad \lambda(f_0)=\frac12\lambda(h)=\frac12\chi_h.$$
 	Thus
 	$$\frac{\kappa(f_0)}{\lambda(f_0)}=\lambda_h=12.$$

Let $p_1=f_0(F_1),\cdots,p_s=f_0(F_s)$.
Then $\{p_1,\cdots,p_s\}$ are just the ramification divisor of the induced double cover
$$j|_{B}\colon  B \lra C.$$
In particular, $s$ is even and $s\geq 2$ since we have assume that $s\neq 0$.
Let $$\pi_n\colon  C_n\lra C$$
be a  cyclic cover of degree $2n+1$ branched exactly over $\{p_1,\cdots,p_s\}$ with ramification indices being all equal to $2n+1$.
Such a cyclic cover exists.
Indeed, take positive integers $0< m_i <2n+1$, $1\leq i \leq s$, such that
\begin{enumerate}[(i).]
	\item $\gcd(m_i,2n+1)=1$ for any $1\leq i \leq s$;
	\item $\sum\limits_{i=1}^{s}m_i$ is a multiple of $2n+1$.
\end{enumerate}
Then the divisor $R:=\sum\limits_{i=1}^{s}m_ip_i$ is $(2n+1)$-divisible, i.e., there exists
a line bundle $L$ such that $\mathcal{O}(R) \sim L^{\otimes (2n+1)}$, where `$\sim$' stands for the linear equivalence.
Then the relation $\mathcal{O}(R) \sim L^{\otimes (2n+1)}$ defines a cyclic cover $\pi_n\colon  C_n\lra C$
satisfying our requirements.
Using such a cyclic cover $\pi_n$ to do the base change,
let $f_n\colon  X_n\to C_n$ be the pullback fibration of $f_0$,
and $F_{i,n} \subseteq X_n$ be the corresponding fiber of $F_i$ for $1\leq i \leq s$.
Then  by Example \ref{exa:FhFg},
$$F_{i,n} \cong F_i \cong F_g^h,\qquad \forall~1\leq i \leq s.$$
Hence by \eqref{timesd}, \eqref{modinv} and \eqref{eqnchern},
$$\begin{aligned}
K_{f_n}^2&\,=\kappa(f_n)+\sum_{i=1}^{s} c_1^2(F_{i,n})=(2n+1)\kappa(f_0)+(2g-2)s,\\
\chi_{f_n}&\,=\lambda(f_n)+\sum_{i=1}^{s}\chi_{F_{i,n}}=(2n+1)\lambda(f_0)+\frac12gs.
\end{aligned}$$
Thus the slope of $f_n$ is
$$\lambda_{f_n}=\frac{K_{f_n}^2}{\chi_{f_n}}=\frac{(2n+1)\kappa(f_0)+(2g-2)s}{(2n+1)\lambda(f_0)+\frac12gs}.
$$
When $n$ goes to infinity,  the slope of $f_n$ tends to $12$ as required.
Moreover, the vertical fundamental group $\mathcal{V}_{f_n}$ is trivial by Lemma \ref{lem-2-2},
since $f_n$ admits a singular fiber isomorphic to $F^h_g$, which is simply connected by Example \ref{exa:FhFg}.
\end{proof}

Applying the above construction to the case $g=3$, we can prove Theorem \ref{thm2}.

\begin{proof}[Proof of Theorem \ref{thm2}]
	We follow the notations introduced in the proof of Theorem \ref{thm1}.
	Take $g=3$, and let $h\colon X \to B$ be the Kodaira fibration of genus $g=3$ constructed in the proof of Theorem \ref{thm1}.
	Let $f_0\colon X_0 \to C\cong B/\langle\tau\rangle$ be the relatively minimal model of
	the quotient fibred surface $X/\langle\sigma\rangle \to C \cong B/\langle\tau\rangle$.
	Then it is known that
	$$\begin{aligned}
		K_{f_0}^2&\,=\kappa(f_0)+\sum_{i=1}^{s} c_1^2(F_i)=\kappa(f_0)+4s,\\
		\chi_{f_0}&\,=\lambda(f_0)+\sum_{i=1}^{s}\chi_{F_i}=\lambda(f_0)+\frac32s,
	\end{aligned}$$
	and that
	$$\kappa(f_0)=\frac12\kappa(h)=\frac12K_h^2,
	\qquad \lambda(f_0)=\frac12\lambda(h)=\frac12\chi_h,
	\qquad \text{with~}~\frac{K_h^2}{\chi_h}=12.$$
	Here $s$ is the number of singular fibers of $f_0$.
	Let $\delta_i~(i=0,1)$ be the divisor class of $\Delta_i$ of $\overline{\CM_3}\backslash{\CM}_3$, and $\mathfrak{h}$ be the divisor class of $\overline{\CH_3}^{DM}$.  By the formulas (see \cite{HM82})
	$$\kappa=\frac13\delta_0+3\delta_1+\frac43\mathfrak{h},~~\lambda=\frac19\delta_0+\frac13\delta_1+\frac19\mathfrak{h},$$
	it follows that
	$$K_h^2=\kappa(h)=\frac{4}{3}s,\qquad \chi_h=\lambda(h)=\frac{1}{9}s.$$
	Hence
	\begin{equation}\label{eqn-3-1}
		\kappa(f_0)=\frac{2}{3}s,\qquad \lambda(f_0)=\frac{1}{18}s,
	\end{equation}
	and
	\begin{equation}\label{eqn-3-2}
		K_{f_0}^2=\kappa(f_0)+4s=\frac{14}{3}s,\qquad \chi_{f_0}=\lambda(f_0)+\frac32s=\frac{14}{9}s.
	\end{equation}
	We claim that
	\begin{claim}\label{claim1}
		For any rational number $r\in \big[3,12)$, there exists a base change $\pi\colon \tilde{C} \to C$, such that
		\begin{enumerate}[(i).]
			\item the slope of $\tilde f$ is $\lambda_{\tilde f}=r$,
			where $\tilde f\colon \tilde X \to \tilde C$ is the pullback fibration under the base change $\pi$;
			\item the fibration $\tilde f$ admits singular fibers, and every singular fiber is isomorphic to $F^h_3$, where $F^h_3$ is the singular hyperelliptic fiber of genus $3$ as in Example \ref{exa:FhFg}.
		\end{enumerate}		
	\end{claim}
	We first prove Theorem \ref{thm2} based on the above claim.
	By the above claim, for any given rational number $r\in \big[3,12)$,
	one can construct a non-hyperelliptic fibration $f_r$ of genus $g=3$,
	such that the slope $\lambda_{f_r}=r$ and the vertical fundamental group is trivial by Lemma \ref{lem-2-2}.
	Combining this with Remark \ref{remark}, we prove Theorem \ref{thm2}.
	It remains to prove the above claim.
	
	{\it \noindent Proof of Claim \ref{claim1}.}
	The main idea is more or less the same as that in the proof of Theorem \ref{thm1}.
	But we have to carefully choose the branched points of the base change
	instead of what we do in the proof of Theorem \ref{thm1}, where we simply take a base change totally ramified over the images of all singular fibers of $f_0$.
	
	First, if $r=3$, then the fibration $f_0$ already satisfies our requirements by \eqref{eqn-3-2}.
	Thus we will assume in the following that $3<r <12$.
	Let $m$ be any odd integer satisfying $m>\frac{27r-72}{12-r}$,
	and let 	
	$$d=\frac{28m(r-3)}{9(3r-8)(m-1)}.$$
	By replacing $f_0$ by a base change unbranched over $\{p_1,\cdots,p_s\}$, we may assume that $ds$ is also an integer at least $2$,
	where $p_i=f_0(F_i)$ is the image of the singular fiber $F_i$.
	Indeed, let $\pi_k:C_k \to C$ be any base change of degree $k$, whose branch locus does not contain any $p_i$, then the number of singular fibers as well as the invariants $K_{f_k}^2$ and $\chi_{f_k}$ are multiplied by $k$.
	Note that $0<d<1$ by definition.
	Let $$\pi\colon \tilde C \to C$$ be a cyclic cover of degree $m$ branched exactly over $\{p_1,\cdots, p_{ds}\}$ with ramification indices being all equal to $m$.
	Such a cyclic cover exists as already showed in the proof of Theorem \ref{thm1}.
	Let $\tilde f$ be the pullback fibration.
	Then the number of singular fibers of $\tilde f$ is $ds+m(s-ds)$ and each singular fiber is isomorphic to $F^h_3$ since $m$ is assumed to be odd.
	Based on \eqref{eqn-3-1} and the formulas in Section\,\ref{sec-modular}, one computes the relative invariants of $\tilde f$ as follows
	$$\begin{aligned}
		K_{\tilde f}^2&\,=m\kappa(f_0)+4\big(ds+m(s-ds)\big)=\Big(\frac{14}{3}m+4d(1-m)\Big)s,\\[2pt]
		\chi_{\tilde f}&\,=m\lambda(f_0)+\frac32\big(ds+m(s-ds)\big)=\Big(\frac{14}{9}m+\frac{3}{2}d(1-m)\Big)s.
	\end{aligned}$$
	So the slope of $\tilde f$ is
	\begin{align*}
		\lambda_{\tilde f}=\frac{K_{\tilde f}^2}{\chi_{\tilde f}}
		=\frac{84m+72(1-m)d}{28m+27(1-m)d}=r.
	\end{align*}
	The last equality follows from the definition of $d$ above.
	Moreover, by construction $\tilde f$ admits singular fibers and every singular fiber is isomorphic to $F^h_3$ as required.
	This proves Claim \ref{claim1}, and hence completes the proof of Theorem \ref{thm2}.
\end{proof}

\begin{remark}\label{rmk-gen}
  Regarding our result in Theorem \ref{thm2} for the case $g=3$,
  we like to conjecture that, for each $g>3$ and each rational number $r\in [4(g-1)/g, 12)$,  there exists a fibration $f$ of genus $g$ such that  $\lambda_f=r$ and $\mathcal{V}_f=\{1\}$.
\end{remark}

%\section*{Acknowledgement}
%
%The authors would like to thank Prof. Sheng-Li Tan, Mu-Lin Li, Yi Gu for helpful discussions.  The authors would like to express their  appreciation to the referee  for many crucial suggestions. %31:thank to referee.
%
%\section*{References}

{Xiao-Lei Liu, School of Mathematical Sciences, Dalian University of Technology, Dalian, Liaoning Province, P. R. of China.

 {\it E-mail address}: xlliu1124@dlut.edu.cn}\vspace{3mm}

{Xin Lu, School of Mathematical Sciences,  Key Laboratory of MEA(Ministry of Education) \& Shanghai Key Laboratory of PMMP,  East China Normal University, Shanghai 200241, China

 {\it E-mail address}:
 xlv@math.ecnu.edu.cn}

\clearpage

\end{document}